\documentclass[a4paper,12pt,reqno]{amsart}

\usepackage{amsmath,amsfonts,amssymb,amsthm,enumerate,multicol}
\usepackage{mathtools}
\usepackage{tikz,graphicx}
\usepackage{float,wrapfig}
\usepackage{epstopdf}
\usepackage{enumitem,wasysym}
\usepackage[labelsep=space,font=footnotesize]{caption}
\usepackage[english]{babel}
\usepackage[autostyle]{csquotes}
\usepackage{autobreak}
\usepackage{xcolor}
\usepackage[colorlinks=true,
            citecolor=green!60!black,
            linkcolor=blue,
            urlcolor=blue]{hyperref}
\usepackage{cleveref}

\newtheorem{definition}{Definition}[section]
\newtheorem{theorem}{Theorem}[section]

\newtheorem{lemma}{Lemma}[section]
\newtheorem{example}{Example}[section]
\newtheorem{remark}{Remark}[section]

\theoremstyle{plain}
\newtheorem{thm}{Theorem}[section]

\newtheorem{prop}[thm]{Proposition}

\allowdisplaybreaks
\numberwithin{equation}{section}

\newcommand{\N}{\mathbb{N}}

\newcommand{\Min}{M_{1}^{\mathrm{in}}}
\newcommand{\win}{\omega_{i}^{\mathrm{in}}}
\newcommand{\Tgel}{T_{\mathrm{gel}}}
\newcommand{\1}{\mathbf{1}}

\title[Gelation and Positivity for the Discrete OHS Coagulation Equation]{Gelation and Positivity of Solutions to the Discrete Oort--Hulst--Safronov Coagulation Equation}

\author{Mashkoor Ali}
\address{Jindal Global Business School, O.P. Jindal Global University,
Sonipat--131001, Haryana, India}
\email{mashkoor.ali@jgu.edu.in}



\begin{document}

\begin{abstract}
Motivated by the recent deterministic approach of Fournier~\cite{F2025}
to gelation for the continuous Smoluchowski coagulation equation, we
adapt his method to the discrete Oort--Hulst--Safronov (OHS) coagulation
system.  We show that under a suitable condition on the coagulation
kernel, every solution with finite initial mass loses mass in finite
time, and we give an explicit bound on the gelation time.  We also prove gelation in the critical logarithmic case and provide a sufficient condition for mass conservation.  Finally, we study the positivity of solutions and show that, for any positive time, a cluster size has positive concentration if and only if it is at least as large
as the smallest cluster present initially.
\end{abstract}

\maketitle
 
\section{Introduction}
\label{sec:intro}

In~\cite{PBD1 1999,PBD 1999}, Dubovski\v{i} investigated a dispersed particulate system and proposed the discrete Oort--Hulst--Safronov (OHS) coagulation model (also known as the Safronov--Dubovski\v{i} coagulation model), which is the discrete analogue of the continuous Oort--Hulst--Safronov (OHS) coagulation model introduced in~\cite{OH 1946,SV 1972}. The continuous OHS model has been extensively studied in~\cite{BP 2007,BPA 2022,PBD1 1999,LLW 2003,PL 2005,PL 2006}, where the issues of well-posedness and various qualitative properties of solutions are rigorously addressed.

In the discrete model, only binary collisions between clusters can occur
simultaneously, and the mass of each cluster is assumed to be
proportional to some $m_{0}>0$, which is the smallest cluster in the
system.  A cluster with mass $im_{0}$ is called an \emph{$i$-mer}.
When a $j$-mer ($j\leq i$) collides with an $i$-mer, the $j$-mer splits
into $j$ monomers, one of which merges with the $i$-mer to produce an
$(i+1)$-mer, while the remaining $j-1$ monomers are released back into
the system.  Denoting by $\omega_{i}(t)\geq 0$ the concentration of
$i$-mers at time $t\geq 0$ and by $\Lambda_{i,j}=\Lambda_{j,i}\geq 0$
the collision rate between $i$-mers and $j$-mers, this mechanism yields
the \emph{discrete OHS coagulation equations} as
\begin{align}
\frac{d\omega_{i}(t)}{dt}
  &= \omega_{i-1}(t)\sum_{j=1}^{i-1}j\,\Lambda_{i-1,j}\,\omega_{j}(t)
   - \omega_{i}(t)\sum_{j=1}^{i}j\,\Lambda_{i,j}\,\omega_{j}(t)
   - \sum_{j=i}^{\infty}\Lambda_{i,j}\,\omega_{i}(t)\,\omega_{j}(t),
   \quad i\in\N,
   \label{eq:DOHSE}\\
\omega_{i}(0)&=\omega_{i}^{\mathrm{in}}\geq 0,\quad i\in\N.
   \label{eq:IC}
\end{align}
The three terms on the right-hand side of~\eqref{eq:DOHSE} encode
distinct physical mechanisms.  The first term describes the creation of
an $i$-mer when an $(i-1)$-mer absorbs one of the $j$ monomers released
by a fragmenting $j$-mer; it is absent when $i=1$.  The second term
accounts for the loss of an $i$-mer that absorbs a monomer (released by
a fragmenting $j$-mer with $j\leq i$) and thereby becomes an $(i+1)$-mer.
The factor $j$ multiplying $\Lambda_{i,j}$ in both of these sums reflects
the number of monomers available from each fragmenting $j$-mer, and is a
structural feature distinctive to the Safronov-Dubovski\v{i} model.
The third term represents the destruction of an $i$-mer when it collides
with a larger $j$-mer ($j\geq i$) and itself breaks into $i$ monomers.

The \emph{total mass} of the system is
$M_{1}(t):=\sum_{i=1}^{\infty}i\,\omega_{i}(t)$.
A formal computation, which we carry out in detail in
Section~\ref{sec:massgel}, suggests that $M_{1}(t)=\Min:=\sum_{i=1}^{\infty}i\,\win$
for all $t\geq 0$, i.e., mass is conserved.  However, when
$\Lambda_{i,j}$ grows sufficiently rapidly for large $i,j$,
infinite-mass clusters may form in finite time, causing
$M_{1}(t)<\Min$; this is \emph{gelation}, and the first time at
which mass is lost is the \emph{gelation time}
$\Tgel:=\inf\{t\geq 0:M_{1}(t)<\Min\}$.

The gelation problem for coagulation equations has a rich history. 
For a comprehensive mathematical treatment of both the continuous
Smoluchowski coagulation equation and its discrete counterparts,
we refer to~\cite{AD 1999, BLL 2019, DC2015,RLD1972}; for the stochastic
counterpart, see~\cite{AIR2023, R2013} and the references therein. For the discrete OHS system specifically, the well-posedness of solutions and the gelation phenomenon were investigated in~\cite{MRG 2024,AG 2023,BAG 2005,DAS 2022,DAV 2014,KKC 2023}.

The present work adapts Fournier's approach~\cite{F2025} to the discrete
system~\eqref{eq:DOHSE}--\eqref{eq:IC}, providing a short, self-contained
deterministic proof of gelation.  The key novelty in the discrete setting
lies in the structure of the weak formulation.  For the Smoluchowski
equation the flux coefficient $\Delta\psi(x,y)=\psi(x+y)-\psi(x)-\psi(y)$
is symmetric in $x$ and $y$, reflecting the symmetric merge of two
clusters.  By contrast, the asymmetric collision mechanism
of~\eqref{eq:DOHSE} yields a flux coefficient
$\Phi_{\psi}(i,j)=(\psi_{i+1}-\psi_{i})\,j-\psi_{j}$ (defined for
$i\geq j$), which is inherently asymmetric.  Despite this difference,
the sign properties of $\Phi_{\psi}$ for the test sequence
$\psi_{i}=i\wedge a$ are sufficient to derive the same scale-by-scale
energy estimate that drives the gelation argument in~\cite{F2025}.  Our
main result (Theorem~\ref{thm:main}) provides, under an explicit
summability condition on the kernel, a finite upper bound on $\Tgel$
valid for all solutions simultaneously.\\
The paper is organized as follows.  Section~\ref{sec:prelim} fixes
notation and the solution concept, and derives the weak formulation.
Section~\ref{sec:massgel} establishes the mass bound and the main
finite-time gelation theorem, with an explicit estimate on the
gelation time and illustrative kernel classes.
Section~\ref{sec:criticality} treats the critical logarithmic case,
proving gelation for $\alpha>1$ and a complementary condition for
mass conservation.  Finally, Section~\ref{sec:positivity} shows that
the positivity set is a half-line determined by the minimal initial
cluster size.

\section{Preliminaries}
\label{sec:prelim}

For a sequence $(\omega_{i})_{i\geq 1}$ of nonnegative real numbers
and $k\geq 0$, the \emph{$k$-th moment} is
$M_{k}[\omega]:=\sum_{i=1}^{\infty}i^{k}\omega_{i}$, and we write
$M_{k}(t):=M_{k}[\omega(t)]$ and
$M_{k}^{\mathrm{in}}:=M_{k}[\omega^{\mathrm{in}}]$.  In particular,
$M_{1}(t)=\sum_{i=1}^{\infty}i\,\omega_{i}(t)$ is the total mass and
$M_{0}(t)=\sum_{i=1}^{\infty}\omega_{i}(t)$ is the total cluster number.

\begin{definition}\label{def:solution}
A family $(\omega_{i}(t))_{i\in\N,\,t\geq 0}$ of nonnegative real
numbers is a \emph{solution} to~\eqref{eq:DOHSE}--\eqref{eq:IC} if,
for each $i\in\N$, the map $t\mapsto\omega_{i}(t)$ is locally absolutely
continuous on $[0,\infty)$; for every $T>0$ and $i\in\N$,
\[
   \int_{0}^{T}\!\!\left[\omega_{i}(s)\sum_{j=1}^{i}j\,\Lambda_{i,j}
   \,\omega_{j}(s)+\omega_{i}(s)\sum_{j=i}^{\infty}\Lambda_{i,j}
   \,\omega_{j}(s)\right]ds<\infty;
\]
equation~\eqref{eq:DOHSE} holds for Lebesgue-almost every $t\geq 0$
and every $i\in\N$; and $\omega_{i}(0)=\win$ for all $i\in\N$.
\end{definition}

The integrability condition ensures that every sum on the right-hand
side of~\eqref{eq:DOHSE} is absolutely convergent in
$L^{1}_{\mathrm{loc}}([0,\infty))$ for each fixed $i$, while the
absolute continuity of $t\mapsto\omega_{i}(t)$ permits integration in
time and interchange of finite sums with time-derivatives by the
dominated convergence theorem.

The following identity, which encodes the full dynamics
of~\eqref{eq:DOHSE} in terms of a single weighted-sum equation for
any bounded test sequence, is the central analytical tool of the paper.

\begin{lemma}\label{lem:weak}
Let $(\omega_{i}(t))$ be a solution in the sense of
Definition~\ref{def:solution}, and let $(\psi_{i})_{i\geq 1}$ be any
bounded sequence of real numbers.  Then for all $t\geq 0$,
\begin{equation}\label{eq:weak}
   \sum_{i=1}^{\infty}\psi_{i}\,\omega_{i}(t)
   =\sum_{i=1}^{\infty}\psi_{i}\,\win
   +\int_{0}^{t}\sum_{\substack{i,j\geq 1\\i\geq j}}
   \Phi_{\psi}(i,j)\,\Lambda_{i,j}\,\omega_{i}(s)\,\omega_{j}(s)\,ds,
\end{equation}
where the \emph{discrete flux coefficient}, defined for $i\geq j\geq 1$,
is
\begin{equation}\label{eq:Phi-def}
   \Phi_{\psi}(i,j):=(\psi_{i+1}-\psi_{i})\,j-\psi_{j}.
\end{equation}
\end{lemma}

\begin{proof}
Multiply~\eqref{eq:DOHSE} by $\psi_{i}$, sum over $i\geq 1$, and
denote by $I_{1}$, $I_{2}$, $I_{3}$ the contributions of the three
terms on the right-hand side.  Since $(\psi_{i})$ is bounded and the
integrability condition of Definition~\ref{def:solution} holds, the
interchange of the sum with the time-derivative is justified after
integration in $t$.

For $I_{1}$, the $i=1$ term vanishes because the inner sum
$\sum_{j=1}^{0}$ is empty.  For $i\geq 2$, the substitution $k:=i-1$
(so $i=k+1$, with $k$ ranging from $1$ to $\infty$ and $j$ from $1$
to $k$) followed by the renaming $k\to i$ yields
\begin{equation}\label{eq:I1}
   I_{1}(t)
   =\sum_{k=1}^{\infty}\psi_{k+1}\,\omega_{k}(t)
   \sum_{j=1}^{k}j\,\Lambda_{k,j}\,\omega_{j}(t)
   =\sum_{\substack{i,j\geq 1\\i\geq j}}
   \psi_{i+1}\,j\,\Lambda_{i,j}\,\omega_{i}(t)\,\omega_{j}(t).
\end{equation}
The double sum in $I_{2}$ runs over $i\geq 1$ and $1\leq j\leq i$,
which is the region $\{i\geq j\geq 1\}$, so
\begin{equation}\label{eq:I2}
   I_{2}(t)
   =-\sum_{\substack{i,j\geq 1\\i\geq j}}
   \psi_{i}\,j\,\Lambda_{i,j}\,\omega_{i}(t)\,\omega_{j}(t).
\end{equation}
For $I_{3}$, the sum runs over $\{j\geq i\geq 1\}$.  Using the
symmetry $\Lambda_{i,j}=\Lambda_{j,i}$, we exchange the summation
indices $i\leftrightarrow j$: the region $\{j\geq i\geq 1\}$ maps to
$\{i\geq j\geq 1\}$, and the summand
$\psi_{i}\,\Lambda_{i,j}\,\omega_{i}\,\omega_{j}$ becomes
$\psi_{j}\,\Lambda_{j,i}\,\omega_{j}\,\omega_{i}
=\psi_{j}\,\Lambda_{i,j}\,\omega_{i}\,\omega_{j}$.  Therefore
\begin{equation}\label{eq:I3}
   I_{3}(t)
   =-\sum_{\substack{i,j\geq 1\\j\geq i}}
   \psi_{i}\,\Lambda_{i,j}\,\omega_{i}\,\omega_{j}
   =-\sum_{\substack{i,j\geq 1\\i\geq j}}
   \psi_{j}\,\Lambda_{i,j}\,\omega_{i}(t)\,\omega_{j}(t).
\end{equation}
Adding~\eqref{eq:I1},~\eqref{eq:I2}, and~\eqref{eq:I3} gives
\begin{equation*}
   I_{1}+I_{2}+I_{3}
   =\sum_{\substack{i\geq j\geq 1}}
   \bigl[\psi_{i+1}\,j-\psi_{i}\,j-\psi_{j}\bigr]
   \Lambda_{i,j}\,\omega_{i}\,\omega_{j}
   =\sum_{\substack{i\geq j\geq 1}}
   \Phi_{\psi}(i,j)\,\Lambda_{i,j}\,\omega_{i}(t)\,\omega_{j}(t).
\end{equation*}
Integrating from $0$ to $t$ gives~\eqref{eq:weak}.
\end{proof}

\begin{remark}\label{rem:flux}
The flux coefficient $\Phi_{\psi}(i,j)=(\psi_{i+1}-\psi_{i})\,j-\psi_{j}$
is \emph{asymmetric} in $i$ and $j$, in contrast to the Smoluchowski
flux $\Delta\psi(x,y)=\psi(x+y)-\psi(x)-\psi(y)$, which is symmetric.
The term $(\psi_{i+1}-\psi_{i})\,j$ captures the gain from the $j$
monomers each capable of merging with an $i$-mer to produce an
$(i+1)$-mer, while $-\psi_{j}$ records the loss of one $j$-cluster
through fragmentation.  Despite this structural difference, the sign
properties of $\Phi_{\psi}$ for the test sequence $\psi_{i}=i\wedge a$
are sufficient to carry out the entire gelation argument.
\end{remark}

\section{Mass Bounds and Finite-Time Gelation}
\label{sec:massgel}

We first verify formally that mass should be conserved, and identify
the precise mechanism by which the conservation argument can fail.
Multiplying~\eqref{eq:DOHSE} by $i$ and summing over $i\geq 1$,
the contributions of the three terms are
\[
   S_{1}=\sum_{i\geq j\geq 1}(i+1)\,j\,\Lambda_{i,j}\,\omega_{i}\omega_{j},
   \quad
   S_{2}=\sum_{i\geq j\geq 1}i\,j\,\Lambda_{i,j}\,\omega_{i}\omega_{j},
\]
where $S_{1}$ is obtained by the substitution $k:=i-1$ followed by
the renaming $k\to i$.  Their difference is
\[
   S_{1}-S_{2}=\sum_{i\geq j\geq 1}j\,\Lambda_{i,j}\,\omega_{i}\omega_{j}.
\]
For $S_{3}$, the sum runs over $\{j\geq i\geq 1\}$; the index exchange
$i\leftrightarrow j$ together with symmetry $\Lambda_{i,j}=\Lambda_{j,i}$
gives
\[
   S_{3}=\sum_{j\geq i\geq 1}i\,\Lambda_{i,j}\,\omega_{i}\omega_{j}
   =\sum_{i\geq j\geq 1}j\,\Lambda_{i,j}\,\omega_{i}\omega_{j}.
\]
Hence $\frac{d}{dt}M_{1}=(S_{1}-S_{2})-S_{3}=0$ formally.  The
cancellation relies on the validity of the index exchange, which requires
absolute convergence of the double sum.  When $\Lambda_{i,j}$ grows too
rapidly, this fails and gelation occurs.  The following proposition
shows rigorously that, regardless of the kernel, $M_{1}$ can never
increase.

\begin{prop}\label{prop:mass}
For any solution $(\omega_{i}(t))$ to~\eqref{eq:DOHSE}--\eqref{eq:IC}
with $\Min<\infty$, we have $M_{1}(t)\leq\Min$ for all $t\geq 0$.
\end{prop}

\begin{proof}
Fix $N\in\N$ and apply Lemma~\ref{lem:weak} with the bounded test
sequence $\psi_{i}:=i\wedge N:=\min(i,N)$.  Since
$\psi_{i+1}-\psi_{i}=\1_{i<N}$, the flux coefficient is
$\Phi_{\psi}(i,j)=j\,\1_{i<N}-(j\wedge N)$.
For $i\geq j\geq 1$ there are two cases: if $j\leq i<N$, then
$\Phi_{\psi}(i,j)=j-j=0$; if $i\geq N$, then
$\Phi_{\psi}(i,j)=-(j\wedge N)\leq 0$.
In both cases $\Phi_{\psi}(i,j)\leq 0$, so the integral
in~\eqref{eq:weak} is nonpositive, giving
\[
   \sum_{i=1}^{\infty}(i\wedge N)\,\omega_{i}(t)
   \leq\sum_{i=1}^{\infty}(i\wedge N)\,\win\leq\Min.
\]
Since $i\wedge N\nearrow i$ as $N\to\infty$ for each fixed $i$, the
Monotone Convergence Theorem yields $M_{1}(t)\leq\Min$.
\end{proof}

We now state and prove the main result.  Fix $a_{0}\in\N$ and $r>1$.
For every integer $a\geq a_{0}$, define the \emph{block intensity}
\begin{equation}\label{eq:Ha}
   \Delta(a):=a\,\inf\bigl\{\Lambda_{i,j}:a\leq i,\,j\leq\lfloor ra\rfloor\bigr\}
\end{equation}
and the \emph{block occupation measure}
\begin{equation*}
   \Pi_{t}(a):=\sum_{i=a}^{\lfloor ra\rfloor}\omega_{i}(t).
\end{equation*}
The quantity $\Delta(a)$ captures the effective coagulation strength over
the block $[a,\lfloor ra\rfloor]^{2}$, normalized by $a$; this
normalization arises from the structure of~\eqref{eq:DOHSE} since
$j\wedge a=a$ when $j\geq a$ in the weak formulation.  The function
$\Pi_{t}(a)$ is the discrete analogue of the occupation measure
$f_{t}([a,ra])$ appearing in the continuous theory~\cite{F2025}.

\begin{theorem}\label{thm:main}
Let $(\Lambda_{i,j})_{i,j\geq 1}$ be a symmetric, nonnegative
coagulation kernel.  Assume there exist $a_{0}\in\N$ and $r>1$ such
that $H(a)>0$ for all integers $a\geq a_{0}$, and
\begin{equation}\label{eq:H2}
   \ell:=\sum_{a=a_{0}}^{\infty}[\Delta(a)]^{-1/2}<\infty.
\end{equation}
Let $(\omega_{i}(t))_{t\geq 0}$ be any solution
to~\eqref{eq:DOHSE}--\eqref{eq:IC} satisfying $\Min<\infty$ and
\begin{equation*}
   \mu_{0}:=\sum_{i=a_{0}+1}^{\infty}(i-a_{0})\,\win>0.
\end{equation*}
Then gelation occurs in finite time, and the gelation time satisfies
\begin{equation}\label{eq:Tgel-bound}
   \Tgel\leq\frac{2\ell^{2}\!\left(\dfrac{r}{r-1}\right)^{\!2}\Min}
   {\mu_{0}^{2}}.
\end{equation}
\end{theorem}

\begin{proof}
We fix an integer $a\geq a_{0}$ and apply Lemma~\ref{lem:weak} with
the test sequence $\psi_{i}:=i\wedge a$.  As in the proof of
Proposition~\ref{prop:mass}, $\psi_{i+1}-\psi_{i}=\1_{i<a}$, so the
flux coefficient reads $\Phi_{\psi}(i,j)=j\,\1_{i<a}-(j\wedge a)$.
For $i\geq j\geq 1$ this vanishes when $j\leq i<a$, and equals
$-(j\wedge a)$ when $i\geq a$.  Consequently the only nonzero
contributions to the sum in~\eqref{eq:weak} come from pairs with
$i\geq a$, and the identity~\eqref{eq:weak} gives
\begin{equation*}
   \sum_{i=1}^{\infty}(i\wedge a)\,\omega_{i}(t)
   -\sum_{i=1}^{\infty}(i\wedge a)\,\win
   =-\int_{0}^{t}\sum_{\substack{i\geq j\geq 1\\i\geq a}}
   (j\wedge a)\,\Lambda_{i,j}\,\omega_{i}(s)\,\omega_{j}(s)\,ds.
\end{equation*}
The integrand on the right is nonnegative, so rearranging yields
\begin{equation}\label{eq:ineq-a}
   \int_{0}^{t}\sum_{\substack{i\geq j\geq 1\\i\geq a}}
   (j\wedge a)\,\Lambda_{i,j}\,\omega_{i}(s)\,\omega_{j}(s)\,ds
   \leq\sum_{i=1}^{\infty}(i\wedge a)\,\win\leq\Min.
\end{equation}
We now extract a lower bound by restricting the domain of summation to
the smaller region $\{a\leq j\leq i\leq\lfloor ra\rfloor\}$, which is
contained in $\{i\geq j\geq 1,\,i\geq a\}$.  On this region, $j\geq a$
implies $j\wedge a=a$, and $i,j\in[a,\lfloor ra\rfloor]$ implies
$a\,\Lambda_{i,j}\geq \Delta(a)$ by definition~\eqref{eq:Ha}, i.e.,
$\Lambda_{i,j}\geq \Delta(a)/a$.  Therefore
\begin{align}
   \int_{0}^{t}\sum_{\substack{i\geq j\geq 1\\i\geq a}}
   (j\wedge a)\,\Lambda_{i,j}\,\omega_{i}\,\omega_{j}\,ds
   &\geq\int_{0}^{t}\sum_{\substack{a\leq j\leq i\\i\leq\lfloor ra\rfloor}}
   a\,\Lambda_{i,j}\,\omega_{i}\,\omega_{j}\,ds
   \geq\int_{0}^{t}\sum_{\substack{a\leq j\leq i\\i\leq\lfloor ra\rfloor}}
   \Delta(a)\,\omega_{i}\,\omega_{j}\,ds.
   \label{eq:lb}
\end{align}
For the triangular sum over the block ${a\leq j\leq i\leq\lfloor ra\rfloor}$, we use the following elementary inequality for any finite collection of nonnegative real numbers $(x_{j})_{j\in S}$,
\begin{equation}\label{eq:sym-bound}
   \sum_{\substack{j,k\in S\\j\geq k}}x_{j}\,x_{k}
   \geq\frac{1}{2}\!\left(\sum_{j\in S}x_{j}\right)^{\!2}.
\end{equation}
Indeed, since $x_{j}^{2}+x_{k}^{2}\geq 2x_{j}x_{k}$ for all
$j,k$, one has
\[
   2\sum_{\substack{j,k\in S\\j\geq k}}x_{j}x_{k}
   =\sum_{j\in S}x_{j}^{2}+2\sum_{\substack{j,k\in S\\j>k}}x_{j}x_{k}
   +\sum_{j\in S}x_{j}^{2}
   \geq\sum_{j\in S}x_{j}^{2}+2\sum_{\substack{j,k\in S\\j>k}}x_{j}x_{k}
   =\left(\sum_{j\in S}x_{j}\right)^{\!2},
\]
and dividing by $2$ gives~\eqref{eq:sym-bound}.  Applying this with
$S=\{a,a+1,\ldots,\lfloor ra\rfloor\}$ and $x_{i}=\omega_{i}(s)$ yield
\[
   \sum_{\substack{a\leq j\leq i\\i\leq\lfloor ra\rfloor}}
   \omega_{i}(s)\,\omega_{j}(s)
   \geq\frac{1}{2}[\Pi_{s}(a)]^{2}.
\]
Inserting into~\eqref{eq:lb} and combining with~\eqref{eq:ineq-a}, we
obtain the \emph{key energy estimate}
\begin{equation}\label{eq:key-energy}
   \int_{0}^{t}\Delta(a)\,[\Pi_{s}(a)]^{2}\,ds\leq 2\Min,
   \qquad\text{for all }t\geq 0\text{ and }a\geq a_{0}.
\end{equation}
This controls the weighted occupation intensity of every geometric
block $[a,\lfloor ra\rfloor]$, uniformly in time.

Having established~\eqref{eq:key-energy} at each scale $a$, we integrate
over all scales.  Multiplying~\eqref{eq:key-energy} by $[H(a)]^{-1/2}$
and summing over $a=a_{0},a_{0}+1,\ldots$ using hypothesis~\eqref{eq:H2}, we obtain
\begin{equation*}
   \int_{0}^{t}\sum_{a=a_{0}}^{\infty}[\Delta(a)]^{1/2}[\Pi_{s}(a)]^{2}\,ds
   \leq 2\ell\,\Min.
\end{equation*}
To connect the left-hand side to the total mass, we decompose
$M_{1}(s)=\Upsilon_{1,s}+\Upsilon_{2,s}$, where
\begin{equation*}
   \Upsilon_{1,s}:=\sum_{i=1}^{\infty}(i\wedge a_{0})\,\omega_{i}(s),
   \qquad
   \Upsilon_{2,s}:=\sum_{i=a_{0}+1}^{\infty}(i-a_{0})\,\omega_{i}(s),
\end{equation*}
and one readily verifies $\Upsilon_{1,s}+\Upsilon_{2,s}=M_{1}(s)$.  The
estimate~\eqref{eq:ineq-a} applied at $a=a_{0}$ gives
$\Upsilon_{1,s}\leq \Upsilon_{1,0}:=\sum_{i=1}^{\infty}(i\wedge a_{0})\,\win$
for all $s\geq 0$.

It remains to relate $\Upsilon_{2,s}$ to $\sum_{a\geq a_{0}}\Pi_{s}(a)$.  For
each $i>a_{0}$, the integer $i$ belongs to the block $[a,\lfloor ra\rfloor]$
if and only if $a\leq i$ and $a\geq i/r$, so that the valid values of
$a$ are precisely those in $[\max(a_{0},\lceil i/r\rceil),\,i]$.  Hence
\begin{align*}
   \sum_{a=a_{0}}^{\infty}\1_{i\in[a,\lfloor ra\rfloor]}
   &=i-\max\!\bigl(a_{0},\lceil i/r\rceil\bigr)+1
   \geq i-\max(a_{0},\,i/r+1)+1
   =i-\max(a_{0}-1,\,i/r) \\
   &=\min\!\bigl(i-a_{0}+1,\,{\textstyle\frac{r-1}{r}}i\bigr)
   \geq\frac{r-1}{r}(i-a_{0}),
\end{align*}
where the last inequality holds in both cases: if the minimum equals
$i-a_{0}+1\geq i-a_{0}$, this is $\geq\frac{r-1}{r}(i-a_{0})$ since
$\frac{r-1}{r}<1$; if the minimum equals $\frac{r-1}{r}i$, this is
$\geq\frac{r-1}{r}(i-a_{0})$ since $a_{0}>0$.  Exchanging the order
of summation, we have
\begin{align*}
   \sum_{a=a_{0}}^{\infty}\Pi_{s}(a)
   &=\sum_{i=a_{0}}^{\infty}\omega_{i}(s)
   \sum_{a=a_{0}}^{\infty}\1_{i\in[a,\lfloor ra\rfloor]}
   \geq\sum_{i=a_{0}+1}^{\infty}\omega_{i}(s)\cdot
   \frac{r-1}{r}(i-a_{0})
   =\frac{r-1}{r}\,\Upsilon_{2,s}.
\end{align*}
Applying the Cauchy-Schwarz inequality with
$u_{a}=[\Delta(a)]^{-1/4}$ and $v_{a}=[\Delta(a)]^{1/4}\Pi_{s}(a)$ gives
\begin{align}
   \frac{r-1}{r}\,\Upsilon_{2,s}
   &\leq\sum_{a=a_{0}}^{\infty}\Pi_{s}(a)
   =\sum_{a=a_{0}}^{\infty}[\Delta(a)]^{-1/4}\cdot[\Delta(a)]^{1/4}\Pi_{s}(a)
   \notag\\
   &\leq\left(\sum_{a=a_{0}}^{\infty}[\Delta(a)]^{-1/2}\right)^{1/2}
   \!\!\left(\sum_{a=a_{0}}^{\infty}[\Delta(a)]^{1/2}[\Pi_{s}(a)]^{2}
   \right)^{1/2}
   =\ell^{1/2}\!\left(\sum_{a=a_{0}}^{\infty}[\Delta(a)]^{1/2}
   [\Pi_{s}(a)]^{2}\right)^{1/2}.
   \label{eq:CS}
\end{align}
Squaring both sides of~\eqref{eq:CS}, we arrive
\begin{equation}\label{eq:Bs-sq}
   \Upsilon_{2,s}^{2}\leq\left(\frac{r}{r-1}\right)^{2}
   \ell\sum_{a=a_{0}}^{\infty}[\Delta(a)]^{1/2}[\Pi_{s}(a)]^{2}.
\end{equation}
Integrating~\eqref{eq:Bs-sq} from $s=0$ to $s=t$ and applying
Fubini's theorem (all terms are non-negative) to interchange the sum and the
time integral yields
\begin{align*}
   \int_{0}^{t}\Upsilon_{2,s}^{2}\,ds
   &\leq\left(\frac{r}{r-1}\right)^{2}\ell
   \sum_{a=a_{0}}^{\infty}[\Delta(a)]^{1/2}\int_{0}^{t}[\Pi_{s}(a)]^{2}\,ds.
\end{align*}
The key energy estimate~\eqref{eq:key-energy} gives
$\Delta(a)\int_{0}^{t}[\Pi_{s}(a)]^{2}\,ds\leq 2\Min$, i.e.,
$[\Delta(a)]^{1/2}\int_{0}^{t}[\Pi_{s}(a)]^{2}\,ds\leq[\Delta(a)]^{-1/2}\cdot 2\Min$.
Summing over $a$ using hypothesis~\eqref{eq:H2} gives
\begin{equation}\label{eq:time-int-final}
   \int_{0}^{t}\Upsilon_{2,s}^{2}\,ds
   \leq 2\left(\frac{r}{r-1}\right)^{2}\ell^{2}\,\Min.
\end{equation}

The right-hand side is independent of $t$, and therefore the time integral of
$\Upsilon_{2,s}^2$ is uniformly bounded by the constant
$2\left(\frac{r}{r-1}\right)^2\ell^2 M_1^{\mathrm{in}}$.
To conclude, suppose for contradiction that $\Tgel=+\infty$, i.e.,
$M_{1}(s)=\Min$ for all $s\geq 0$.  From $M_{1}(s)=\Upsilon_{1,s}+\Upsilon_{2,s}$ and
the bound $\Upsilon_{1,s}\leq \Upsilon_{1,0}$, we have
\[
   \Upsilon_{2,s}=M_{1}(s)-\Upsilon_{1,s}\geq\Min-\Upsilon_{1,0}
   =\sum_{i=a_{0}+1}^{\infty}(i-a_{0})\,\win=\mu_{0}>0.
\]
Inserting $\Upsilon_{2,s}\geq\mu_{0}$ into~\eqref{eq:time-int-final} yields
\[
   T\,\mu_{0}^{2}\leq\int_{0}^{T}\Upsilon_{2,s}^{2}\,ds
   \leq 2\left(\frac{r}{r-1}\right)^{2}\ell^{2}\,\Min,
\]
which must hold for every $T>0$.  Since the right-hand side is a fixed
finite constant, this is a contradiction for sufficiently large $T$.
Therefore $\Tgel<\infty$, and the bound~\eqref{eq:Tgel-bound} follows
by taking $T=\Tgel$ and rearranging.
\end{proof}

We now verify condition~\eqref{eq:H2} for natural families of kernels
and derive explicit gelation-time bounds from Theorem~\ref{thm:main}.
Throughout we write $i\vee j:=\max(i,j)$ and $i\wedge j:=\min(i,j)$.

\begin{example}\label{ex:kernels}
We list several kernel classes for which hypothesis~\eqref{eq:H2} holds,
together with the corresponding bounds on $\Delta(a)$ and $\ell$.
\end{example}

\medskip
\begin{enumerate}
\item \textbf{Kernels of the form $\Lambda_{i,j}\geq
C_{0}(i\wedge j)^{\gamma}$, $\gamma>1$.}
For $i,j\in[a,\lfloor ra\rfloor]$ one has $i\wedge j\geq a$, so
$\Lambda_{i,j}\geq C_{0}a^{\gamma}$ and by definition~\eqref{eq:Ha},
$\Delta(a)\geq C_{0}a^{1+\gamma}$, giving
$[\Delta(a)]^{-1/2}\leq C_{0}^{-1/2}a^{-(1+\gamma)/2}$.  Since $\gamma>1$
implies $(1+\gamma)/2>1$, comparison with the integral
$\int_{a_{0}}^{\infty}a^{-(1+\gamma)/2}\,da
=\dfrac{2\,a_{0}^{(1-\gamma)/2}}{\gamma-1}$ yields
\[
   \ell\leq\frac{2\,C_{0}^{-1/2}a_{0}^{(1-\gamma)/2}}{\gamma-1}<\infty,
   \qquad
   \Tgel\leq\frac{8\,C_{0}^{-1}a_{0}^{1-\gamma}}{(\gamma-1)^{2}}
   \cdot\left(\frac{r}{r-1}\right)^{2}\cdot\frac{\Min}{\mu_{0}^{2}}.
\]
\medskip
\item \textbf{Power-law kernels with ratio factor.}
Let $\theta\geq 0$ and $\gamma>1$.  For
\[
   \Lambda_{i,j}=\rho(i\wedge j)^{\gamma}
   \left(\frac{i\wedge j}{i\vee j}\right)^{\!\theta},
\]
one has for $i,j\in[a,\lfloor ra\rfloor]$ with any $r>1$ that
$(ij)^{\gamma/2}\geq a^{\gamma}$ and
$(i\wedge j)^{\gamma}(i\wedge j/i\vee j)^{\theta}\geq r^{-\theta}a^{\gamma}$
respectively, giving $\Delta(a)\geq\rho r^{-\theta}a^{1+\gamma}$  and
\[
   \ell\leq\frac{2(\rho r^{-\theta})^{-1/2}a_{0}^{(1-\gamma)/2}}{\gamma-1}
   <\infty,
   \qquad
   \Tgel\leq\frac{8(\rho r^{-\theta})^{-1}a_{0}^{1-\gamma}}{(\gamma-1)^{2}}
   \cdot\left(\frac{r}{r-1}\right)^{2}\cdot\frac{\Min}{\mu_{0}^{2}}.
\]

\medskip
\item \textbf{Truncated-ratio kernels.}
Fix $\varepsilon>0$ and $\theta\geq 0$.  For
\[
   \Lambda_{i,j}
   =(i\wedge j)
   \left(\frac{i\wedge j}{i\vee j}-\frac{1}{2}\right)_{\!+}^{\theta}
   \log^{2+\varepsilon}(e+i\wedge j),
\]
with $r\in(1,2)$. Since $i\vee j\leq ra<2a\leq 2(i\wedge j)$, one has
$(i\wedge j/i\vee j-1/2)_{+}\geq 1/r-1/2>0$, giving
\[
   \Delta(a)\geq\left(\frac{1}{r}-\frac{1}{2}\right)^{\!\theta}
   a^{2}\log^{2+\varepsilon}(e+a),
   \qquad
   [\Delta(a)]^{-1/2}\leq\frac{(1/r-1/2)^{-\theta/2}}
   {a\log^{1+\varepsilon/2}(e+a)}.
\]
Since $1+\varepsilon/2>1$, the series converges by the integral test, and comparison with
$\int_{a_{0}}^{\infty}[x\log^{1+\varepsilon/2}(e+x)]^{-1}dx$ yields
\[
   \ell\leq\frac{4\,(1/r-1/2)^{-\theta/2}}
   {\varepsilon\,\log^{\varepsilon/2}(e+a_{0})}<\infty,
   \quad
   \Tgel\leq\frac{32}{\varepsilon^{2}(1/r-1/2)^{\theta}
   \log^{\varepsilon}(e+a_{0})}
   \left(\frac{r}{r-1}\right)^{2}\frac{\Min}{\mu_{0}^{2}},
\]
with $r$ restricted to $(1,2)$ exactly as in~\cite{F2025}.

\item \textbf{One-homogeneous kernels with logarithmic correction.}
Fix $\varepsilon>0$ and $\theta\geq 0$.  For
\[
   \Lambda_{i,j}
   =(i\wedge j)\left(\frac{i\wedge j}{i\vee j}\right)^{\!\theta}
   \log^{2+\varepsilon}(e+i\wedge j),
\]
and $i,j\in[a,\lfloor ra\rfloor]$ with any $r>1$, one has
$i\wedge j\geq a$ and $(i\wedge j)/(i\vee j)\geq 1/r$, so
\[
   \Delta(a)\geq r^{-\theta}a^{2}\log^{2+\varepsilon}(e+a),
   \qquad
   [\Delta(a)]^{-1/2}\leq\frac{r^{\theta/2}}{a\log^{1+\varepsilon/2}(e+a)}.
\]
Since $1+\varepsilon/2>1$, the series converges by the integral test, and comparison with
$\int_{a_{0}}^{\infty}[x\log^{1+\varepsilon/2}(e+x)]^{-1}dx$ yields
\[
   \ell \leq\frac{4\,r^{\theta/2}}{\varepsilon\,\log^{\varepsilon/2}(e+a_{0})}<\infty,
   \qquad
   \Tgel\leq\frac{32\,r^{\theta}}{\varepsilon^{2}\log^{\varepsilon}(e+a_{0})}
   \left(\frac{r}{r-1}\right)^{2}\frac{\Min}{\mu_{0}^{2}}.
\]
Unlike the truncated-ratio kernel, any $r>1$ is admissible here, since
$\Lambda_{i,j}$ does not vanish for well-separated clusters.

\item \textbf{Mixed polynomial-logarithmic kernels.}
Let $\gamma>1$, $\theta\geq 0$, and $\varepsilon\geq 0$.  For
\begin{equation}\label{eq:mixed-kernel}
   \Lambda_{i,j}
   \geq C_{0}(i\wedge j)^{\gamma}
   \left(\frac{i\wedge j}{i\vee j}\right)^{\!\theta}
   \log^{\varepsilon}(e+i\wedge j),
\end{equation}
one has $\Delta(a)\geq C_{0}r^{-\theta}a^{1+\gamma}\log^{\varepsilon}(e+a)$
and $[\Delta(a)]^{-1/2}\leq(C_{0}r^{-\theta})^{-1/2}
a^{-(1+\gamma)/2}\log^{-\varepsilon/2}(e+a)$.
Since $\gamma>1$ gives $(1+\gamma)/2>1$, this is summable for every
$\varepsilon\geq 0$ (the logarithmic factor only accelerates
convergence), and comparison with
$\int_{a_{0}}^{\infty}a^{-(1+\gamma)/2}\,da$ yields
\[
   \ell \leq\frac{2\,(C_{0}r^{-\theta})^{-1/2}a_{0}^{(1-\gamma)/2}}
   {\gamma-1},
   \qquad
   \Tgel\leq\frac{8\,C_{0}^{-1}r^{\theta}a_{0}^{1-\gamma}}{(\gamma-1)^{2}}
   \left(\frac{r}{r-1}\right)^{2}\frac{\Min}{\mu_{0}^{2}}.
\]
In the boundary case $\gamma=1$ the conclusion holds provided
$\varepsilon>2$, since then $[H(a)]^{-1/2}\leq(C_{0}r^{-\theta})^{-1/2}
[a\log^{\varepsilon/2}(e+a)]^{-1}$ is summable iff $\varepsilon/2>1$,
recovering case~\textbf{(4)} above.  The kernels
in~\eqref{eq:mixed-kernel} with $\theta>0$ are not covered by any
earlier work on the discrete OHS system, yet Theorem~\ref{thm:main}
shows that gelation is robust to the suppression of coagulation
between clusters of very different sizes, provided only that $\gamma>1$.
\end{enumerate}
By Theorem~\ref{thm:main}, every solution with $\Min<\infty$ and
$\mu_{0}>0$ gels in finite time for all kernels in
cases~\textbf{(1)}--\textbf{(5)}.

We next treat the factorized class and symmetric polynomial kernels, which allow us to write $\kappa$ and $\Tgel$ in fully explicit form.

\begin{prop}\label{prop:factorized}
Let $\theta:\N\to(0,\infty)$ be a non-decreasing sequence and set
$\Lambda_{i,j}:=\theta_{i}\theta_{j}$.  Then $H(a)\geq a\theta_{a}^{2}$
for any $a_{0}\in\N$ and $r>1$, and hypothesis~\eqref{eq:H2} holds if
and only if
\begin{equation}\label{eq:factorized-cond}
   \ell=\sum_{a=a_{0}}^{\infty}\frac{1}{\sqrt{a}\;\theta_{a}}<\infty.
\end{equation}
The same conclusion holds for the class
$\Lambda_{i,j}=\theta_{i}\theta_{j}+\kappa_{i,j}$ with
$0\leq\kappa_{i,j}\leq A\theta_{i}\theta_{j}$ and $A\geq 0$.  In
particular, for $\theta_{i}=C^{1/2}i^{\gamma/2}$ with $C>0$ and
$\gamma>1$, condition~\eqref{eq:factorized-cond} gives
\begin{align}\label{eq:kappa-factorized}
   \ell\leq\frac{2C^{-1/2}a_{0}^{(1-\gamma)/2}}{\gamma-1}<\infty,
   \qquad
   \Tgel\leq\frac{8C^{-1}a_{0}^{1-\gamma}}{(\gamma-1)^{2}}
   \cdot\left(\frac{r}{r-1}\right)^{2}\cdot\frac{\Min}{\mu_{0}^{2}}.
\end{align}
\end{prop}

\begin{proof}
Since $\theta$ is non-decreasing, $\theta_{i}\geq\theta_{a}$ and
$\theta_{j}\geq\theta_{a}$ for $i,j\in[a,\lfloor ra\rfloor]$, giving
$\Lambda_{i,j}\geq\theta_{a}^{2}$ and hence
$\Delta(a)\geq a\theta_{a}^{2}$.  Thus
$[\Delta(a)]^{-1/2}\leq(\sqrt{a}\;\theta_{a})^{-1}$, and summing over
$a\geq a_{0}$ gives~\eqref{eq:factorized-cond}.  For the perturbed
class, $\Lambda_{i,j}\geq\theta_{i}\theta_{j}$ so $\Delta(a)$ is unchanged.
The power-law case $\theta_{i}=C^{1/2}i^{\gamma/2}$ gives
$\Delta(a)\geq Ca^{1+\gamma}$ and the bound~\eqref{eq:kappa-factorized}
follows from~\eqref{eq:Tgel-bound}.
\end{proof}

The factorized class in Proposition~\ref{prop:factorized} extends
well beyond power laws.  For instance, $\theta_{i}=i(\log(e+i))^{s}$
with $s\ge0$ satisfies~\eqref{eq:factorized-cond} since
$\sum_{a\geq a_{0}}(a^{3/2}(\log(e+a))^{s})^{-1}<\infty$ for any
$s\ge0$.

\begin{prop}\label{prop:polynomial}
Let $\alpha,\beta>0$ with $\alpha\geq\beta$ and let
$\Lambda_{i,j}=i^{\alpha}j^{\beta}+i^{\beta}j^{\alpha}$.  Then
$\Delta(a)\geq 2a^{1+\alpha+\beta}$ for any $a_{0}\in\N$ and $r>1$, and
hypothesis~\eqref{eq:H2} holds if and only if $\alpha+\beta>1$, in
which case
\begin{align}\label{eq:kappa-polynomial}
   \ell\leq\frac{2^{1/2}a_{0}^{(1-\alpha-\beta)/2}}{\alpha+\beta-1}
   <\infty,
   \qquad
   \Tgel\leq\frac{4\left(\dfrac{r}{r-1}\right)^{2}
   a_{0}^{1-\alpha-\beta}\,\Min}{(\alpha+\beta-1)^{2}\,\mu_{0}^{2}}.
\end{align}
\end{prop}

\begin{proof}
For $i,j\in[a,\lfloor ra\rfloor]$, one has $i\geq a$ and $j\geq a$, so
$i^{\alpha}j^{\beta}+i^{\beta}j^{\alpha}\geq 2a^{\alpha+\beta}$ and
$\Delta(a)\geq 2a^{1+\alpha+\beta}$.  Hence
$[\Delta(a)]^{-1/2}\leq 2^{-1/2}a^{-(1+\alpha+\beta)/2}$, which is
summable iff $\alpha+\beta>1$.  Comparing with the integral gives
the bound on $\ell$ in~\eqref{eq:kappa-polynomial}, and substituting
into~\eqref{eq:Tgel-bound} yields the bound on $\Tgel$.
\end{proof}

The threshold $\alpha+\beta>1$ in Proposition~\ref{prop:polynomial} is
sharp within our framework, since the series $\sum a^{-1}$ diverges at
$\alpha+\beta=1$.  As a canonical illustration, for the multiplicative
kernel $\Lambda_{i,j}=ij$ (i.e., $\alpha=\beta=1$), choosing $a_{0}=1$
and $r=2$ gives $\ell\leq\sqrt{2}$ and $(r/(r-1))^{2}=4$, so
\begin{align}\label{eq:Tgel-mult}
   \Tgel\leq 16\cdot\frac{\Min}{\mu_{0}^{2}}.
\end{align}
We record a monotonicity property of the gelation time with respect
to the initial data.  Let $(\omega_{i}^{(k),\mathrm{in}})_{i\geq 1}$,
$k=1,2$, be two initial data with respective total masses
$M_{1}^{(k),\mathrm{in}}$ and tail moments
$\mu_{0}^{(k)}:=\sum_{i>a_{0}}(i-a_{0})\omega_{i}^{(k),\mathrm{in}}>0$.
Since the constants $\kappa$ and $r$ in bound~\eqref{eq:Tgel-bound}
depend only on the kernel, two natural comparison scenarios arise.

\medskip
\noindent\textit{Same tail moment, different total mass.}
If $\mu_{0}^{(1)}=\mu_{0}^{(2)}$ and
$M_{1}^{(1),\mathrm{in}}\leq M_{1}^{(2),\mathrm{in}}$, then
$\Tgel^{(1)}\leq\Tgel^{(2)}$: for the same distribution of large
clusters, more total mass forces earlier gelation.

\medskip
\noindent\textit{Same total mass, different tail moment.}
If $M_{1}^{(1),\mathrm{in}}=M_{1}^{(2),\mathrm{in}}$ and
$\mu_{0}^{(1)}\geq\mu_{0}^{(2)}$, then
$\Tgel^{(1)}\leq\Tgel^{(2)}$: for the same total mass, initial data
more concentrated at large cluster sizes gels sooner.

\medskip
More generally, whenever
\[
   \frac{M_{1}^{(1),\mathrm{in}}}{\bigl(\mu_{0}^{(1)}\bigr)^{2}}
   \leq
   \frac{M_{1}^{(2),\mathrm{in}}}{\bigl(\mu_{0}^{(2)}\bigr)^{2}},
\]
bound~\eqref{eq:Tgel-bound} gives immediately that
$\Tgel^{(1)}\leq\Tgel^{(2)}$.  The ratio $\Min/\mu_{0}^{2}$ thus
serves as a natural measure of how concentrated the initial mass is at
large cluster sizes.  To see this concretely, consider monodisperse
initial data $\omega_{i}^{\mathrm{in}}=m\delta_{in}$ at cluster size
$n\geq a_{0}+1$ with total density $m>0$.  Then $\Min=mn$ and
$\mu_{0}=m(n-a_{0})$, giving
\[
   \frac{\Min}{\mu_{0}^{2}}
   =\frac{n}{m(n-a_{0})^{2}}\to 0
   \quad\text{as }m\to\infty\text{ or }n\to\infty,
\]
confirming that systems initially loaded with many clusters, or with
clusters at large sizes, gel in shorter time.

The gelation problem for the discrete OHS system~\eqref{eq:DOHSE} was studied previously in~\cite{AG 2023} and~\cite{DAS 2022}, and we now explain precisely how Theorem~\ref{thm:main} relates to both works. In~\cite{AG 2023}, gelation is established for kernels of the form
\begin{align}\label{eq:ali-giri}
   \Lambda_{i,j}\geq C(ij)^{\lambda/2},
   \qquad\lambda\in(1,2),\quad C>0,
\end{align}
via moment estimate. The kernel~\eqref{eq:ali-giri} is a special case of Proposition~\ref{prop:factorized} with $\theta_{i}=C^{1/2}i^{\lambda/2}$. Consequently, Theorem~\ref{thm:main} recovers the gelation result of~\cite{AG 2023}, with the additional benefit of providing
the explicit bound
\[
   \Tgel\leq\frac{8C^{-1}(\lambda-1)^{-2}a_{0}^{1-\lambda}
   (r/(r-1))^{2}\Min}{\mu_{0}^{2}},
\]
which was not established in~\cite{AG 2023}.  Moreover, Theorem~\ref{thm:main}
applies to all $\lambda>1$ (not only $\lambda\in(1,2)$).

In~\cite{DAS 2022}, instantaneous gelation ($\Tgel=0$) is proved under
the two-sided bound
\begin{align}\label{eq:das-saha}
   \mathcal{B}_{L}(i^{\alpha}j^{\beta}+i^{\beta}j^{\alpha})
   \leq\Lambda_{i,j}
   \leq \mathcal{B}_{U}(1+i)^{\omega}(1+j)^{\omega},
   \qquad\alpha,\beta\geq 1,\quad\omega>1.
\end{align}
The lower bound in~\eqref{eq:das-saha} implies, for
$i,j\in[a,\lfloor ra\rfloor]$, that $\Lambda_{i,j}\geq 2\mathcal{B}_{L}a^{\alpha+\beta}$
and hence $\Delta(a)\geq 2\mathcal{B}_{L}a^{1+\alpha+\beta}$.  Since
$\alpha,\beta\geq 1$ gives $\alpha+\beta\geq 2>1$,
Proposition~\ref{prop:polynomial} applies and Theorem~\ref{thm:main}
yields $\Tgel<\infty$.  The result of~\cite{DAS 2022} is of course
sharper in this regime ($\Tgel=0$), but Theorem~\ref{thm:main}
provides the complementary quantitative upper bound
\[
   \Tgel\leq\frac{4(r/(r-1))^{2}(2\mathcal{B}_{L})^{-1}
   a_{0}^{1-\alpha-\beta}\Min}{(\alpha+\beta-1)^{2}\mu_{0}^{2}},
\]
and extends the gelation conclusion to all $\alpha,\beta>0$ with
$\alpha+\beta>1$, a range not accessible by the methods of~\cite{DAS 2022}.

To summarize, the complete picture for the kernel
$\Lambda_{i,j}=i^{\alpha}j^{\beta}+i^{\beta}j^{\alpha}$ is as follows.

\begin{itemize}
\item When $\alpha+\beta\leq 1$, our method gives no information;
moreover, gelation is not expected to occur for any initial data,
since the kernel grows too slowly to drive infinite-mass cluster
formation in finite time.

\item When $\alpha+\beta\in(1,2)$, which includes the case
      $\alpha=\beta=\lambda/2$ with $\lambda\in(1,2)$ of~\cite{AG 2023},
      Theorem~\ref{thm:main} gives finite-time gelation with the explicit
      bound~\eqref{eq:kappa-polynomial}.
\item When $\alpha+\beta\geq 2$ with $\alpha,\beta\geq 1$,
the stronger conclusion $\Tgel=0$ was proved in~\cite{DAS 2022},
whereas our result provides a complementary finite upper bound on
$\Tgel$ in terms of the initial data.
\end{itemize}

\section{Non-gelation and the Critical Case}
\label{sec:criticality}

The summability condition~\eqref{eq:H2} is sufficient
but not necessary for gelation.  This can be seen by comparing the
gelation and non-gelation criteria in the critical logarithmic case.
On the one hand, condition~\eqref{eq:H2} applied to a kernel
$\Lambda_{i,j}=(i+j)\log^{\alpha}(e+i\wedge j)$ gives
$\Delta(a)\geq c\,a^{2}\log^{\alpha}(e+a)$ for some $c>0$, so that
$[\Delta(a)]^{-1/2}\leq c^{-1/2}[a\log^{\alpha/2}(e+a)]^{-1}$ is summable if
and only if $\alpha>2$; thus Theorem~\ref{thm:main} yields gelation only
for $\alpha>2$.  On the other hand, Proposition~\ref{prop:log-gel-nocutoff}
establishes gelation for the same kernel for \emph{all} $\alpha>1$, a
strictly larger range.  Hence gelation occurs in the entire regime
$\alpha\in(1,2]$ where~\eqref{eq:H2} fails, which shows that the
summability of $[\Delta(a)]^{-1/2}$ is not necessary for mass loss.  The
gap reflects the fact that Theorem~\ref{thm:main} controls only the
diagonal blocks $[a,\lfloor ra\rfloor]^{2}$ through the local intensity
$\Delta(a)$, whereas the sharper argument of
Proposition~\ref{prop:log-gel-nocutoff} exploits the cumulative
contribution of all pairs $(i,j)$ with $i\geq j$ via a global test
sequence.

\begin{prop}\label{prop:log-gel-nocutoff}
Let $\alpha>1$ and suppose $\Lambda_{i,j}\geq C(i+j)\log^{\alpha}(e+i\wedge j)$
for all $i,j\geq 1$ and some $C>0$.  Then every solution
to~\eqref{eq:DOHSE}--\eqref{eq:IC} with $\Min\in(0,\infty)$ gels in
finite time, with $\Tgel<\infty \quad\text{and}\quad \Tgel\leq \frac{C_{*}}{\Min}$ depending only on $\alpha$ and $C$.
\end{prop}

\begin{proof}
Fix $a\in\N$ and define the bounded nonneg test sequence
\begin{equation*}
   \psi_{a}(i)
   :=i\,\mathbf{1}_{i\leq a}
   \sum_{k=i}^{\infty}\frac{1}{k\log^{\alpha}(e+k)},
   \quad i\geq 1.
\end{equation*}
Since $\alpha>1$ the series converges, so $\psi_a$ is well-defined and
bounded uniformly in $a$.  For $1\leq j\leq i\leq a-1$, the flux
coefficient~\eqref{eq:Phi-def} satisfies
\begin{align}
   \Phi_{\psi_a}(i,j)
   &=j\!\left[\sum_{k=i+1}^{\infty}\frac{1}{k\log^{\alpha}(e+k)}
   -\frac{1}{\log^{\alpha}(e+i)}\right]
   -j\sum_{k=j}^{\infty}\frac{1}{k\log^{\alpha}(e+k)}
   \notag\\
   &=-j\sum_{k=j}^{i}\frac{1}{k\log^{\alpha}(e+k)}
   -\frac{j}{\log^{\alpha}(e+i)},
   \label{eq:flux-nocutoff}
\end{align}
which is strictly negative.  For $i\geq a$ the flux coefficient is
nonpositive and all pairs contribute nonpositively to the sum
in~\eqref{eq:weak}.  We establish a uniform lower bound on
$|\Phi_{\psi_a}(i,j)|$ for $i\geq j\geq 1$ by treating two cases.

\medskip
\noindent\textit{Case~1: $j\leq i<2j$.}
Retaining only the second term in~\eqref{eq:flux-nocutoff} and using
$\log^{\alpha}(e+i)\leq\log^{\alpha}(e+2j)\leq 2^{\alpha}\log^{\alpha}(e+j)$
(since $\log(e+2j)\leq\log 2+\log(e+j)\leq 2\log(e+j)$ for $j\geq 1$):
\begin{equation}\label{eq:Phi-case1}
   -\Phi_{\psi_a}(i,j)
   \geq\frac{j}{\log^{\alpha}(e+i)}
   \geq\frac{j}{\log^{\alpha}(e+2j)}
   \geq\frac{j}{2^{\alpha}\log^{\alpha}(e+j)}.
\end{equation}

\medskip
\noindent\textit{Case~2: $i\geq 2j$.}
Restricting the sum in~\eqref{eq:flux-nocutoff} to $k\in[j,2j]$ (all
summands are positive, so restricting gives a lower bound) and using
$\log^{\alpha}(e+k)\leq\log^{\alpha}(e+2j)\leq 2^{\alpha}\log^{\alpha}(e+j)$
for $k\leq 2j$ yield
\begin{align*}
   -\Phi_{\psi_a}(i,j)
   &\geq j\sum_{k=j}^{2j}\frac{1}{k\log^{\alpha}(e+k)}
   \geq\frac{j}{\log^{\alpha}(e+2j)}\sum_{k=j}^{2j}\frac{1}{k}.
\end{align*}
Comparing with the integral $\int_{j}^{2j+1}dx/x=\log((2j+1)/j)\geq\log 2$ gives
\begin{equation}\label{eq:Phi-case2}
   -\Phi_{\psi_a}(i,j)
   \geq\frac{j\log 2}{\log^{\alpha}(e+2j)}
   \geq\frac{j\log 2}{2^{\alpha}\log^{\alpha}(e+j)}.
\end{equation}
Combining~\eqref{eq:Phi-case1} and~\eqref{eq:Phi-case2}, there exists a
constant $C_{\alpha}:=\log 2/2^{\alpha}>0$ depending only on $\alpha$ such
that
\begin{equation*}
   -\Phi_{\psi_a}(i,j)\geq\frac{C_{\alpha}j}{\log^{\alpha}(e+j)}
   \quad\text{for all }i\geq j\geq 1.
\end{equation*}
Using $i\wedge j=j$ for $i\geq j$ and the kernel lower bound
$\Lambda_{i,j}\geq C(i+j)\log^{\alpha}(e+j)$, we obtain
\begin{align*}
   -\Lambda_{i,j}\,\Phi_{\psi_a}(i,j)
   &\geq C(i+j)\log^{\alpha}(e+j)\cdot\frac{C_{\alpha}j}{\log^{\alpha}(e+j)}
   =CC_{\alpha}j(i+j)
   \geq CC_{\alpha}\,ij,
\end{align*}
where the last inequality uses $j(i+j)=ij+j^{2}\geq ij$.

Applying the weak identity~\eqref{eq:weak} with $\psi_a$, using the
nonnegativity $\sum_i\psi_a(i)\omega_i(t)\geq 0$, and the symmetrization
identity $2\sum_{i\geq j}ij\,\omega_i\omega_j=\sum_{i,j=1}^a ij\,\omega_i\omega_j$, we arrive
\begin{equation}\label{eq:M1sq-nocutoff}
   \frac{CC_{\alpha}}{2}\int_{0}^{t}[M_1^{(a)}(s)]^{2}\,ds
   \leq\sum_{i=1}^{\infty}\psi_a(i)\,\win.
\end{equation}
For the right-hand side of~\eqref{eq:M1sq-nocutoff}, we bound
$\psi_a(i)$ uniformly.  Since $x\mapsto[x\log^{\alpha}(e+x)]^{-1}$ is
decreasing, comparison with the integral $\int_{i}^{\infty}
[x\log^{\alpha}(e+x)]^{-1}dx$ (which converges for $\alpha>1$) gives
\[
   \sum_{k=i}^{\infty}\frac{1}{k\log^{\alpha}(e+k)}
   \leq\frac{C_{\alpha}'}{i\log^{\alpha-1}(e+i)}
\]
for a constant $C_{\alpha}'>0$ depending only on $\alpha$.  Hence
\[
   \psi_a(i)
   \leq\frac{C_{\alpha}'}{\log^{\alpha-1}(e+i)}\leq C_{\alpha}',
\]
using $\log(e+i)\geq 1$.

Therefore
\begin{equation}\label{eq:psi-bound}
   \sum_{i=1}^{\infty}\psi_a(i)\,\win
   \leq C_{\alpha}'\sum_{i=1}^{\infty}i\,\omega_i^{\mathrm{in}}
   =C_{\alpha}'\Min.
\end{equation}
Combining~\eqref{eq:M1sq-nocutoff} and~\eqref{eq:psi-bound}
and letting $a\to\infty$ by monotone convergence gives
\begin{equation*}
   \int_{0}^{t}[M_1(s)]^{2}\,ds
   \leq\frac{2C_{\alpha}'\Min}{CC_{\alpha}}
   =:C_{*}\Min
   \quad\text{for all }t\geq 0.
\end{equation*}
This bound is independent of $t$.  If $\Tgel=+\infty$ then $M_1(s)=\Min$
for all $s\geq 0$, giving $t[\Min]^{2}\leq C_{*}\Min$, i.e.,
$t\Min\leq C_{*}$ for all $t>0$, a contradiction for $t>C_{*}/\Min$.
Hence $\Tgel<\infty$ and $\Tgel\leq C_{*}/\Min$.
\end{proof}

Having established gelation for kernels growing at least like
$(i+j)\log^{\alpha}(e+i\wedge j)$ with $\alpha>1$, we now turn to the
opposite direction and identify a class of slowly growing kernels for
which mass is conserved for all time, thereby identifying the borderline
between the gelling and non-gelling regimes.

\begin{prop}\label{prop:nongel}
Suppose there exists $C_{1}>0$ such that
$\Lambda_{i,j}\leq C_{1}(i+j)\log(e+i\wedge j)$ for all $i,j\geq 1$.
If $\Min+M_{2}^{\mathrm{in}}<\infty$, where
$M_{2}^{\mathrm{in}}:=\sum_{i=1}^{\infty}i^{2}\win$, then every
solution satisfies $M_{1}(t)=\Min$ for all $t\geq 0$.
\end{prop}

\begin{proof}
The argument adapts Steps~1--3 of Proposition~7 in~\cite{F2025} to
the discrete setting.

We begin with a discrete Jensen-type bound.  For any nonneg sequence
$(\omega_{i})$ and any $A\geq M_{1}[\omega]$, applying Jensen's
inequality to the concave function $\log(e+\cdot)$ under the probability
measure $\nu_{N}=Z_{N}^{-1}\sum_{i=1}^{N}i\omega_{i}\delta_{i}$, where
$Z_{N}=\sum_{i=1}^{N}i\omega_{i}$, yields
\begin{equation}\label{eq:disc-Jensen}
   \sum_{i=1}^{N}i\log(e+i)\,\omega_{i}
   \leq Z_{N}\log\!\left(e+Z_{N}^{-1}\sum_{i=1}^{N}i^{2}\omega_{i}\right)
   \leq A\log\!\left(e+A^{-1}\sum_{i=1}^{N}i^{2}\omega_{i}\right),
\end{equation}
where the second inequality uses the fact that $z\mapsto z\log(e+z^{-1}b)$
is nondecreasing in $z>0$ for each fixed $b>0$.

We next show that $\sup_{[0,T]}M_{2}(t)<\infty$ for every $T>0$.
Fix $N\in\N$ and apply Lemma~\ref{lem:weak} with
$\psi_{i}^{(N)}:=i^{2}\mathbf{1}_{i\leq N}$.  For $i\geq j\geq 1$
with $i\leq N-1$, a direct computation of the flux coefficient gives
\begin{align*}
   \Phi_{\psi^{(N)}}(i,j)
   &=\bigl[(i+1)^{2}-i^{2}\bigr]j-j^{2}
   =(2i+1)j-j^{2}
   =2ij+j(1-j)
   \leq 2ij,
\end{align*}
where the last inequality uses $j\geq 1$ so that $j(1-j)\leq 0$.
For $i\geq N$, the flux coefficient is nonpositive and contributes
nothing to the upper bound.  Hence, using the kernel bound
$\Lambda_{i,j}\leq C_{1}(i+j)\log(e+j)$ for $i\geq j$ and the
estimate $ij(i+j)\leq 2i^{2}j$ (valid for $i\geq j\geq 1$) yield
\begin{align}
   \frac{d}{dt}\sum_{i=1}^{N}i^{2}\omega_{i}(t)
   &\leq 4C_{1}\sum_{\substack{i\geq j\geq 1\\i\leq N}}
   i^{2}j\log(e+j)\,\omega_{i}\omega_{j}
   \leq 4C_{1}\!\left(\sum_{i=1}^{N}i^{2}\omega_{i}\right)
   \!\left(\sum_{j=1}^{N}j\log(e+j)\,\omega_{j}\right).
   \label{eq:M2-ode}
\end{align}
Applying~\eqref{eq:disc-Jensen} to the second factor on the right
of~\eqref{eq:M2-ode} with $A=M_{1}(t)\leq\Min$ and writing
$X(t):=\sum_{i=1}^{N}i^{2}\omega_{i}(t)$ gives
\[
   \frac{d}{dt}X(t)
   \leq 4C_{1}\Min\cdot X(t)\log\!\left(e+\frac{X(t)}{\Min}\right).
\]
The Gronwall lemma applied to this differential inequality gives
$\sup_{t\in[0,T]}X(t)\leq C_{T}<\infty$ for every $T>0$, where
$C_{T}$ depends only on $T$, $\Min$, and $M_{2}^{\mathrm{in}}$.
Letting $N\to\infty$ by monotone convergence yields
$\sup_{[0,T]}M_{2}(t)<\infty$.
We finally establish mass conservation.  Apply Lemma~\ref{lem:weak}
with $\psi_{i}^{(N)}:=i\wedge N$.  From the proof of
Proposition~\ref{prop:mass}, the flux coefficient satisfies
$\Phi_{\psi^{(N)}}(i,j)\leq 0$ for all $i\geq j\geq 1$, and
$|\Phi_{\psi^{(N)}}(i,j)|\leq i\wedge j$ uniformly in $N$.
To let $N\to\infty$ via dominated convergence in the integral
of~\eqref{eq:weak}, it suffices to verify that
\[
   \int_{0}^{T}\sum_{i\geq j\geq 1}(i\wedge j)\Lambda_{i,j}
   \omega_{i}(s)\omega_{j}(s)\,ds<\infty.
\]
For $i\geq j$, the kernel bound gives
$(i\wedge j)\Lambda_{i,j}\leq C_{1}j(i+j)\log(e+j)\leq 2C_{1}ij\log(e+j)$.
Since $\log(e+j)\leq 1+j$, we obtain
\[
   \sum_{i\geq j\geq 1}(i\wedge j)\Lambda_{i,j}\omega_{i}\omega_{j}
   \leq 2C_{1}\bigl(M_{1}^{2}+M_{1}M_{2}\bigr)
   \leq 2C_{1}\bigl([\Min]^{2}+\Min M_{2}(s)\bigr),
\]
which is locally integrable in time since $\sup_{[0,T]}M_{2}(t)<\infty$.
Dominated convergence and the monotone limit $i\wedge N\nearrow i$
then give $M_{1}(t)=\Min$ for all $t\geq 0$.
\end{proof}

Our results therefore strictly contain those of~\cite{AG 2023} as a
special case, extend the gelation criterion to the range
$\alpha+\beta\in(1,2)$ which lies outside the scope of~\cite{DAS 2022},
and furthermore cover entirely new classes of kernels, including
kernels with a logarithmic correction of the form
$(i\wedge j)\log^{\alpha}(e+i\wedge j)$ and kernels carrying the ratio
suppression factor $(i\wedge j/i\vee j)^{\theta}$, neither of which
were treated in any prior work on the discrete OHS system.
Finally, we remark that our method requires $\Delta(a)>0$, that is,
$\Lambda_{i,j}>0$ for all $i,j\in[a,\lfloor ra\rfloor]$, and in
particular on the diagonal $i=j$.  Kernels vanishing on the diagonal,
such as $\Lambda_{i,j}=|i-j|^{\gamma}$, lie outside the present
theory; their treatment remains an open problem for the discrete OHS
system, as it does for the continuous Smoluchowski equation
in~\cite{F2025}.

\section{Positivity of Solutions}
\label{sec:positivity}

Having characterised gelation in the preceding sections, we now address
a complementary structural question: for $t>0$, which cluster
concentrations $\omega_{i}(t)$ are strictly positive, and how does this
depend on the initial data?  For the discrete Smoluchowski coagulation equation
this problem was studied by Da~Costa~\cite{DC1995}, who proved that
under the strict positivity assumption $\Lambda_{i,j}>0$ for all $i$ and $j$,
the set $\mathcal{J}(t):=\{j\in\N:\omega_{j}(t)>0\}$ is independent of
$t>0$ and equals the additive monoid
$\mathrm{span}_{\N_{0}}(P)=\{j=\sum_{i}n_{i}p_{i}:p_{i}\in P,\,
n_{i}\in\N_{0},\,\max_{i}n_{i}>0\}$ generated by the initial
support $P=\mathcal{J}(0)$.  This set can be arithmetically complex and
depend sensitively on the structure of $P$: for instance,
$P=\{2,5\}$ gives $\mathrm{span}_{\N_{0}}(P)=\{2,4,5,6,7,8,\ldots\}$,
which excludes all odd integers smaller than $5$.

For the discrete OHS system~\eqref{eq:DOHSE}, the
situation is strikingly simpler.  The gain term in~\eqref{eq:DOHSE}
has a strictly sequential structure: species $i$ can only be produced
from species $i-1$, not from an arbitrary pair of species whose sizes
sum to $i$ as in Smoluchowski coagulation.  This chain structure forces
$\mathcal{J}(t)$ to be a half-line for all $t>0$, regardless of the
detailed shape of $P$ beyond its minimum.  Continuing the example
above, with $P=\{2,5\}$, the discrete OHS system produces
$\mathcal{J}(t)=\{2,3,4,5,6,\ldots\}$ for all $t>0$, so that species
$3$, which is unreachable in the Smoluchowski case, becomes positive
at any positive time.

Throughout this section we impose the strict positivity condition
\begin{equation}\label{eq:pos-kernel}
   \Lambda_{i,j}>0\quad\text{for all }i,j\geq 1,
\end{equation}
and we consider a solution $(\omega_{i}(t))_{i\in\N,\,t\geq 0}$ to
\eqref{eq:DOHSE}--\eqref{eq:IC} in the sense of
Definition~\ref{def:solution}, with $\Min<\infty$ and nontrivial
initial data $\omega^{\mathrm{in}}\neq 0$.  We define the
\emph{positivity set}, the \emph{initial support}, and the
\emph{minimal initial cluster size} by
\begin{equation*}
   \mathcal{J}(t):=\{i\in\N:\omega_{i}(t)>0\},
   \quad
   P:=\mathcal{J}(0)=\{i\in\N:\omega_{i}^{\mathrm{in}}>0\},
   \quad
   p^{*}:=\min P.
\end{equation*}

The analysis rests on the \emph{variation-of-constants} representation
obtained by rewriting~\eqref{eq:DOHSE} as the first-order linear ODE
$\dot{\omega}_{i}(t)=R_{i}(t)-\varphi_{i}(t)\,\omega_{i}(t)$, where
the \emph{total loss rate} and the \emph{gain rate} are respectively
\begin{equation}\label{eq:phi-R}
   \varphi_{i}(t)
   :=\sum_{j=1}^{i}j\,\Lambda_{i,j}\,\omega_{j}(t)
    +\sum_{j=i}^{\infty}\Lambda_{i,j}\,\omega_{j}(t),
   \qquad
   R_{i}(t)
   :=\begin{cases}
     \displaystyle
     \omega_{i-1}(t)\sum_{j=1}^{i-1}j\,\Lambda_{i-1,j}\,\omega_{j}(t)
     & i\geq 2,\\[6pt]
     0 & i=1.
     \end{cases}
\end{equation}
Setting $E_{i}(t):=\exp\!\bigl(\int_{0}^{t}\varphi_{i}(s)\,ds\bigr)$,
which satisfies $E_{i}(t)\geq 1$ and $E_{i}(t)<\infty$ for all finite
$t$ by Definition~\ref{def:solution}. Next, a straightforward calculation yields
\begin{equation}\label{eq:variation}
   \omega_{i}(t)\,E_{i}(t)
   =\omega_{i}^{\mathrm{in}}
   +\int_{0}^{t}E_{i}(s)\,R_{i}(s)\,ds,
   \qquad i\in\N,\quad t\geq 0.
\end{equation}
Since $E_{i}(t),R_{i}(t)\geq 0$ and $\omega_{i}^{\mathrm{in}}\geq 0$,
identity~\eqref{eq:variation} recovers the nonnegativity of solutions.
Moreover, since $E_{i}(t)>0$, we have $\omega_{i}(t)>0$ if and only
if the right-hand side of~\eqref{eq:variation} is strictly positive.

A key structural consequence of assumption~\eqref{eq:pos-kernel} is
the following equivalence, which has no counterpart in the Smoluchowski
model, we have
\begin{equation}\label{eq:Ri-equiv}
   R_{i}(t)>0\;\iff\;\omega_{i-1}(t)>0,
   \qquad i\geq 2,\quad t\geq 0.
\end{equation}
To see this, note that if $\omega_{i-1}(t)>0$ then the term $j=i-1$
contributes $(i-1)\,\Lambda_{i-1,i-1}\,\omega_{i-1}(t)>0$ to the
sum in~\eqref{eq:phi-R}, giving $R_{i}(t)>0$; conversely, if
$\omega_{i-1}(t)=0$ then the product $\omega_{i-1}(t)\cdot(\cdots)=0$
forces $R_{i}(t)=0$.  Equivalence~\eqref{eq:Ri-equiv} is the engine
of the entire positivity argument: whether species $i$ is being produced
at time $t$ depends solely on whether species $i-1$ is present, and
not on any arithmetic combination of smaller species as in the
Smoluchowski gain term.  We now prove the main result through a series
of lemmas.

\begin{lemma}\label{lem:monotone}
For all $t\geq\tau\geq 0$, $\mathcal{J}(t)\supseteq\mathcal{J}(\tau)$.
\end{lemma}

\begin{proof}
Let $i\in\mathcal{J}(\tau)$, so $\omega_{i}(\tau)>0$.  Since $R_{i}\geq 0$,
identity~\eqref{eq:variation} gives, for all $t\geq\tau$,
\[
   \omega_{i}(t)E_{i}(t)
   =\omega_{i}^{\mathrm{in}}+\int_{0}^{t}E_{i}(s)R_{i}(s)\,ds
   \geq\omega_{i}^{\mathrm{in}}+\int_{0}^{\tau}E_{i}(s)R_{i}(s)\,ds
   =\omega_{i}(\tau)E_{i}(\tau)>0,
\]
so $\omega_{i}(t)>0$ and thus $i\in\mathcal{J}(t)$.
\end{proof}

\begin{lemma}\label{lem:below-pstar}
For every $i<p^{*}$ and every $t\geq 0$, $\omega_{i}(t)=0$.
\end{lemma}

\begin{proof}
If $p^{*}=1$ there is nothing to prove, so assume $p^{*}\geq 2$.
We argue by induction on $i=1,\ldots,p^{*}-1$.  For the base case
$i=1$: since $1<p^{*}$ implies $1\notin P$, we have
$\omega_{1}^{\mathrm{in}}=0$; moreover $R_{1}\equiv 0$ by
definition~\eqref{eq:phi-R}.  Hence~\eqref{eq:variation} gives
$\omega_{1}(t)E_{1}(t)=0$ for all $t\geq 0$, and since $E_{1}(t)>0$,
$\omega_{1}(t)=0$ for all $t\geq 0$.  For the induction step, assume
$\omega_{k}(t)=0$ for all $t\geq 0$ and all $k\leq i-1$, where
$i<p^{*}$.  In particular $\omega_{i-1}(t)=0$ for all $t$, so
$R_{i}(t)=0$ for all $t$ by~\eqref{eq:phi-R}.  Since also
$\omega_{i}^{\mathrm{in}}=0$ (as $i<p^{*}$ implies $i\notin P$),
identity~\eqref{eq:variation} gives $\omega_{i}(t)E_{i}(t)=0$, hence
$\omega_{i}(t)=0$ for all $t\geq 0$.
\end{proof}

\begin{lemma}\label{lem:propagation}
Under assumption~\eqref{eq:pos-kernel}, $\omega_{i}(t)>0$ for all
$i\geq p^{*}$ and all $t>0$.
\end{lemma}

\begin{proof}
We proceed by induction on $i\geq p^{*}$.  For the base case
$i=p^{*}$: since $p^{*}\in P$ we have $\omega_{p^{*}}^{\mathrm{in}}>0$,
and since $R_{p^{*}}\geq 0$, identity~\eqref{eq:variation} gives
\[
   \omega_{p^{*}}(t)E_{p^{*}}(t)
   \geq\omega_{p^{*}}^{\mathrm{in}}>0
   \quad\text{for all }t\geq 0.
\]
Since $E_{p^{*}}(t)<\infty$, we conclude $\omega_{p^{*}}(t)>0$ for
all $t\geq 0$, and in particular for all $t>0$.

For the induction step, let $i>p^{*}$ and suppose $\omega_{i-1}(s)>0$
for all $s>0$.  Fix any $t>0$.  By the induction hypothesis,
$\omega_{i-1}(s)>0$ for all $s\in(0,t)$, and by
equivalence~\eqref{eq:Ri-equiv} this gives $R_{i}(s)>0$ for all
$s\in(0,t)$.  Since $E_{i}(s)\geq 1>0$, the integrand
$E_{i}(s)R_{i}(s)$ is strictly positive on the open interval $(0,t)$,
and hence
\[
   \omega_{i}(t)E_{i}(t)
   \geq\int_{0}^{t}E_{i}(s)R_{i}(s)\,ds>0.
\]
Since $E_{i}(t)<\infty$ we conclude $\omega_{i}(t)>0$.  As $t>0$ was
arbitrary, the claim holds for all $t>0$, completing the induction.
\end{proof}

\begin{theorem}\label{thm:positivity}
Assume~\eqref{eq:pos-kernel} and let $(\omega_{i}(t))$ be a solution
to~\eqref{eq:DOHSE}--\eqref{eq:IC} with $\Min<\infty$ and
$\omega^{\mathrm{in}}\neq 0$.  Then, for all $t>0$,
\[
   \mathcal{J}(t)=\{i\in\N:i\geq p^{*}\},
\]
and this set is independent of $t$, independent of the coagulation
rates $\Lambda_{i,j}$, and independent of the structure of $P$ beyond
its minimum $p^{*}$.  In particular, $\mathcal{J}(t)$ is always an
infinite set for $t>0$.
\end{theorem}

\begin{proof}
Lemma~\ref{lem:below-pstar} gives $\mathcal{J}(t)\subseteq\{i\in\N:i\geq p^{*}\}$
for all $t\geq 0$, while Lemma~\ref{lem:propagation} gives the
reverse inclusion $\{i\in\N:i\geq p^{*}\}\subseteq\mathcal{J}(t)$ for
all $t>0$.  Together these yield $\mathcal{J}(t)=\{i\in\N:i\geq p^{*}\}$
for all $t>0$.  Since the right-hand side is independent of $t$ and
of $(\Lambda_{i,j})$, the same holds for $\mathcal{J}(t)$.  Finally,
the set $\{i\in\N:i\geq p^{*}\}$ is infinite.
\end{proof}

\begin{remark}
The contrast with the Smoluchowski positivity theorem of Da
Costa~\cite{DC1995} is sharp. In the Smoluchowski
model the characterisation $\mathcal{J}(t)=\mathrm{span}_{\N_{0}}(P)$
depends on the \emph{full arithmetic structure} of $P$: two different
choices of $P$ with the same minimum $p^{*}$ can give rise to
entirely different positivity sets for $t>0$. By contrast,
Theorem~\ref{thm:positivity} shows that in the discrete OHS
system $\mathcal{J}(t)$ is determined by $p^{*}$ alone, independently
of any further arithmetic properties of $P$. 
\end{remark}

\subsection*{Acknowledgements} MA expresses deep gratitude to Jindal Global Business School, O.P. Jindal Global University, for its invaluable support in providing essential resources.

\end{document}